\documentclass[twoside,12pt]{article}
\usepackage{amsmath, amsthm, amscd, amsfonts, amssymb, graphicx, color}
\usepackage{graphicx}
 \usepackage[vmargin=3cm,left=3cm,right=3cm]{geometry}
\begin{document}
\setcounter{page}{1}
\setlength{\unitlength}{12mm}
\newcommand{\f}{\frac}
\newtheorem{theorem}{Theorem}[section]
\newtheorem{lemma}[theorem]{Lemma}
\newtheorem{proposition}[theorem]{Proposition}
\newtheorem{corollary}[theorem]{Corollary}
\theoremstyle{definition}
\newtheorem{definition}[theorem]{Definition}
\newtheorem{example}[theorem]{Example}
\newtheorem{solution}[theorem]{Solution}
\newtheorem{xca}[theorem]{Exercise}
\theoremstyle{remark}
\newtheorem{remark}[theorem]{Remark}
\numberwithin{equation}{section}
\newcommand{\sta}{\stackrel}
\title{\bf Weak fuzzy topology on fuzzy topological vector spaces}
\author { B. Daraby$^{a}$, N. Khosravi$^{b}$, A. Rahimi$^{c}$}
\date{\footnotesize $^{a, b, c}$Department of Mathematics, University of Maragheh, P. O. Box 55136-553, Maragheh, Iran\\}
\maketitle

\begin{abstract}
\noindent
In this paper, we introduce the concept of weak fuzzy linear topology on a fuzzy topological vector space as a generalization of usual weak topology. We prove that this topology consists of all weakly lower semi-continuous fuzzy sets on a given vector space when $ \mathbb{K} $ ($ \mathbb{R} $ or $ \mathbb{C} $) considered with its usual fuzzy topology. In the case that the topology of $ \mathbb{K} $ is different from the usual fuzzy topology, we show that the weak fuzzy topology is not equivalent with the topology of weakly lower semi-continuous fuzzy sets.


\vspace{.5cm}\leftline{{Subject Classification 2010: 03E72, 54A40
}}

\vspace{.5cm}\leftline{Keywords: Fuzzy topology; Weak fuzzy topology; Duality.}
\end{abstract}
\section{Introduction}
Katsaras \cite{k10} introduced the concepts of fuzzy vector space and fuzzy topological vector space. He \cite{k7} considered the usual fuzzy topology on the corresponding scalar field $ \mathbb{K} $ , namely the one consisting of all lower semi-continuous functions from $\mathbb{K}$ into $ I $, where $ I=[0, 1] $. Also, in the studying fuzzy topological vector space he \cite{k8} investigated the idea of fuzzy norm on linear spaces. Since then, many authors like Felbin \cite{g}, Cheng and Mordeson \cite{f}, Bag and Samanta \cite{poi}, Sadeqi and Yaqub Azari \cite{s}, and so on started to introduce the notion of fuzzy normed linear space. Later, Xiao and Zhu \cite{u}, Fang \cite{dD}, Daraby et. al. (\cite{d}, \cite{d2}) redefined, the idea of Felbin's \cite{g} definition of fuzzy norm and studied various properties of its topological structure. Das \cite{kh} constructed a fuzzy topology generated by fuzzy norm and studied some properties of this topology. Fang \cite{dD} investigated a new I-topology $ \mathcal{T}^*_{\Vert.\Vert} $ on the fuzzy normed linear space.

 In this paper, we introduce the concept of weak fuzzy linear topology on a fuzzy topological vector space as a generalization of usual weak topology. Moreover, we prove that the weak fuzzy topology is not equivalent with the topology of weakly lower semi-continuous fuzzy sets. In fact, the weak fuzzy topology on a fuzzy topological vector space is an extension of weak topology. The difference between weak topology and weak fuzzy topology leads to new results on the theory of topological vector spaces.


\section{Preliminaries}
Let $ X $ be a non-empty set. A fuzzy set in $ X $ is an element of the set $ I^X $ of all functions from $ X $ into $ I $.
\begin{definition}\cite{kh} Let $X$ and $Y$ be any two sets, $ f: X \rightarrow Y $ be a mapping and $ \mu $ be a fuzzy subset of $X$. Then $f(\mu)$ is a fuzzy subset of $Y$ defined by
$$f(\mu)(y)=\begin{cases}\sup_{x\in f^{-1}(y)} \mu(x) &\quad f^{-1}(y)\neq\emptyset,\\0 & \:\:\:\: else,\end{cases}$$
for all $y \in Y$, where $f^{-1}(y)=\lbrace x: f(x)=y \rbrace$. If $\eta$ is a fuzzy subset of $Y$, then the fuzzy subset $f^{-1}(\eta)$ of $X$ is defined by $f^{-1}(\eta)(x)=\eta(f(x))$ for all $x\in X$.
\end{definition}
\begin{definition}\cite{k7}
A fuzzy topology on a set $ X $ is a subset $ \varphi $ of $ I^{X} $ satisfying the following conditions:
\begin{enumerate}
\item[$(i)$] $ \varphi $ contains every constant fuzzy set in $ X $,
\item[$(ii)$] if $ \mu_{1},\mu_{2} \in \varphi $, then $ \mu_{1} \wedge \mu_{2} \in \varphi $,
\item[$(iii)$] if $ \mu_{i} \in \varphi $ for each $ i \in A $, then $\sup_{i\in A}\mu_{i} \in \varphi$.
\end{enumerate}
The pair $ (X,\varphi) $ is called a fuzzy topological space.
\end{definition}
The elements of $ \varphi$ are called open fuzzy sets in $ X $.
 \begin{definition}\cite{kh}
 A fuzzy topological space $ (X,\varphi) $ is said to be fuzzy Hausdorff if for $ x,y \in X $ and $ x\neq y $ there exist $ \eta, \beta \in \varphi $ with $ \eta(x)=\beta(y)=1  $ and $ \eta \wedge \beta=0 $.
 \end{definition}
  A mapping $f$ from a fuzzy topological space $X$ to a fuzzy topological space $Y$ is called fuzzy continuous at some point $ x \in X $ if $f^{-1}(\mu)$ is open in $ X $ for each open fuzzy set $\mu$ in $Y$.

 Suppose $X$ is a fuzzy topological space and $x\in X $. A fuzzy set $\mu$ in $ X $ is called a neighborhood of $x\in X$ if there is an open fuzzy set $\eta$ with $ \eta\leq \mu $ and $ \eta(x)=\mu(x)>0 $. Warren \cite{k11} has proved that a fuzzy set $\mu$ in $X$ is open if and only if $\mu$ is a neighborhood of $x$ for each  $ x\in X $ with $\mu(x)>0$. Also, he proved that $ f $ is fuzzy continuous if and only if $ f $ is fuzzy continuous at every $ x\in X $.

 \begin{definition}\cite{kh}
 If $ \mu_{1}$ and $\mu_{2} $ are two fuzzy subsets of a vector space $ E $, then the fuzzy set $ \mu_1+\mu_2 $ is defined by
 $$ (\mu_1+\mu_2)(x)=\sup_{x=x_1+x_2} (\mu_1(x_1)\wedge \mu_2(x_2)). $$
  If $ t \in \mathbb{K}, $ we define the fuzzy sets $ \mu_1 \times \mu_2  $  and $ t\mu $ as follows:
$$ (\mu_1 \times \mu_2)(x_1,x_2)= \min \lbrace \mu_1(x_1), \mu_2(x_2)\rbrace $$
and
\begin{enumerate}
\item[$(i)$] \;\; for $ t\neq 0,\: (t\mu)(x)=\mu(\frac{x}{t}) $ \: for all \: $ x \in E $,
\item[$(ii)$]\:\: for $ t=0, \: (t\mu)(x)= \begin{cases}0 &   x\neq 0, \quad \\ \sup_{y \in E} \mu(y) &x=0. \end{cases} $
\end{enumerate}
 \end{definition}
  A net $(x_\alpha)_{\alpha\in T}$ in a fuzzy topological vector space $E$ is fuzzy convergent to $x$ if and only if for each fuzzy neighborhood $\mu$ of $x$ and each $0<\varepsilon<\mu(x)$ there is $\alpha_o\in T$ such that $\mu(x_\alpha)>\varepsilon$ for all $\alpha>\alpha_0$. A net $(x_\alpha)_{\alpha\in T}$ in a fuzzy topological vector space $E$ is fuzzy Cauchy if and only if for each fuzzy neighborhood $\mu$ of zero and each $0<r<\mu(0)$ there is $\alpha_0\in T$ such that for every $\alpha,\alpha'\in T$ with $\alpha,\alpha'>\alpha_0$, $|\mu(x_\alpha-x_{\alpha'})|>r$. A fuzzy topological vector space $E$ is called complete if and only if each Cauchy net in $E$ is convergent \cite{kh}.

 A fuzzy set $ \mu $ in vector space $ E $ is called balanced if $ t\mu \leq \mu $ for each scalar $ t $ with $ \vert t \vert \leq1$. As is shown in \cite{k10}, $ \mu $ is balanced if and only if $ \mu(tx)\geq \mu(x) $ for each $ x\in E $ and each scalar $ t $ with $ \vert t \vert\leq1 $. Also, when $ \mu $ is balanced, we have $ \mu(0) \geq \mu(x) $ for each $ x\in E $. The fuzzy set $ \mu $ is called absorbing if and only if $ \sup_{t>0}t\mu=1 $. Then a fuzzy set $\mu$ is absorbing whenever $\mu(0)=1$. We shall say that the fuzzy set $\mu$ is convex if and only if for all $t\in I $, $t\mu+(1-t)\mu\leq\mu$ \cite{k7}.

\begin{definition}\cite{dD1}
 A fuzzy topology $\tau$ on a vector space $E$ is said to be a fuzzy vector topology, if the two mappings\\
$ f:E \times E \rightarrow E, (x,y)\rightarrow x+y,$\\
$ g:\mathbb{K} \times E \rightarrow E, (t,x)\rightarrow tx,$\\
are continuous when $ \mathbb{K} $ is equipped with the fuzzy topology induced by the usual topology, $ E \times E $ and $ \mathbb{K} \times E $ are the corresponding product fuzzy topologies. A vector space $ E $ with a fuzzy vector topology $ \tau $, denoted by the pair $ (E,\tau) $ is called fuzzy topological vector space (abbreviated to FTVS).
\end{definition}
\begin{definition}\cite{k7}
Let $(E,\tau)$ be a fuzzy topological vector space the collection $ \nu\subset\tau $ of neighborhoods of zero is a local base whenever for each neighborhood $ \mu $ of zero and each $\theta\in(0,\mu(0))$ there is $ \gamma\in \nu $ such that $ \gamma\leq\mu $ and $ \gamma(0)>\theta $.
\end{definition}
\begin{definition}\cite{k8}
A fuzzy seminorm on $E$ is a fuzzy set $\mu$ in $ E $ which is absolutely convex and absorbing.
\end{definition}
We say that a fuzzy set $\mu$, in a vector space $E$, absorbs a fuzzy set $\eta$ if $\mu(0) > 0$ and for every $\theta <\mu(0)$ there exists $t > 0$ such that $\theta\wedge (t\eta) <\mu$. A fuzzy set $\mu$ in a fuzzy linear space $E$ is called bounded if it is absorbed by every neighborhood of zero. A locally convex fuzzy linear space $E$ is called bornological if every absolutely convex fuzzy set in $E$ which absorbs bounded sets is a neighborhood
of zero or equivalently every fuzzy bounded linear operator from $E$ into any fuzzy topological vector space is fuzzy continuous \cite{k8}.

In this paper, the concepts of fuzzy real numbers and Felbin-fuzzy normed are considered in the sense of Xiao and Zhu which are defined below:
\begin{definition}\cite{u} A mapping $\eta:\mathbb{R} \rightarrow I $ is called a fuzzy
real number (fuzzy intervals), whose $\alpha$-level set is denoted by $[\eta]_{\alpha}$, i.e.,
$[\eta]_{\alpha}=\{t:\eta(t)\geq\alpha\}$, if it satisfies two axioms:
\begin{enumerate}
\item[($N_{1}$)] There exists $r'\in \mathbb{R}$ such that $\eta (r')=1$.
\item[($N_{2}$)] For all $ 0<\alpha\leq 1, $ there exist real numbers
$-\infty<\eta^{-}_{\alpha}\leq \eta^{+}_{\alpha} <+\infty$
such that  $[\eta]_{\alpha}$ is equal to the closed interval
$[\eta^{-}_{\alpha}, \eta^{+}_{\alpha}]$.
\end{enumerate}
\end{definition}
The set of all fuzzy real numbers (fuzzy intervals) is denoted by $F(\mathbb{R})$.
If $\eta\in F(\mathbb{R})$ and $\eta(t)=0$ whenever $t<0$, then $\eta$ is called
a non-negative fuzzy real number and $F^{\ast}(\mathbb{R})$
denotes the set of all non-negative fuzzy real numbers \cite{d}.
\begin{definition}\cite{u}
Let $E$ be a vector space over $ \mathbb{R} $; $L$ and $R$ (respectively, left norm and right norm)
be symmetric and non-decreasing mappings from $I \times I $
into $I$ satisfying $L(0, 0)=0$ and $ R(1, 1)=1.$
Then $\|\cdot\|$ is called a fuzzy norm and $(E, \|\cdot\|, L, R)$
is a fuzzy normed linear space (abbreviated to FNLS)
if the mapping $\|\cdot\|$ from $E$ into $F^{\ast}(\mathbb{R})$
satisfies the following axioms, where $[\|x\|]_{\alpha}=[\|x\|^{-}_{\alpha}, \|x\|^{+}_{\alpha}]$
for $x\in E$ and $\alpha\in (0, 1]:$
\begin{enumerate}
\item[($F1$)]   $x=0$ if and only if $\|x\|=\widetilde{0},$
\item[($F2$)]$\|rx\|=|r|\odot\|x\|$ for all $x\in E$ and $r\in (-\infty, \infty),$
\item[($F3$)] $\forall x, y\in E:$
\item[($F3R$)] if $s\geq\|x\|^{-}_{1}, t\geq\|y\|^{-}_{1}$ and $s+t\geq\|x+y\|^{-}_{1},$ then\\
$\|x+y\|(s+t)\leq R(\|x\|(s), \|y\|(t)),$
\item[($F3L$)] if $s\leq\|x\|^{-}_{1}, t\leq\|y\|^{-}_{1}$ and $s+t\leq\|x+y\|^{-}_{1},$ then\\
$\|x+y\|(s+t)\geq L(\|x\|(s), \|y\|(t)).$
\end{enumerate}
\end{definition}
\begin{definition}\cite{kh}
For fixed $ \alpha \in (0, 1] $ and $ \varepsilon >0 $, the fuzzy set $ \mu_{\alpha}(x,\varepsilon) $ defined in $(E, \|\cdot\|)$ by
 $$ \mu_{\alpha}(x,\varepsilon)(y)=\begin{cases}\alpha \:\: &  \|y-x\|^{+}_{\alpha} < \varepsilon, \quad\\0 &else,\end{cases} $$
is said to be an $ \alpha-open $ sphere in $(E, \|\cdot\|)$.
\end{definition}
\begin{definition}\cite{dD}
Any fuzzy set $ \mu \in I^{E} $ is defined to be $\|.\|-$ linearly open if for every $ x\in supp(\mu)$ and $ \alpha\in (0, \mu(x)) $ there exists $\varepsilon >0 $ such that $ \mu_{\alpha}(x,\varepsilon) \leq \mu. $
\end{definition}
We set $\mathcal{T}^{\ast}_{\|.\|}=\lbrace \mu\in I^{E}\: |\:\: \mu \:\:is\:\:\: \Vert.\Vert-linearly\:\:open \rbrace $. It was proved in \cite{dD} that if $(E, \|\cdot\|, L, R)$ is a fuzzy normed space, then $ (E,\|\cdot\|) $ is a fuzzy topological vector space.\\
The Katsaras norm was defined as follows:
\begin{definition}\cite{k8}
A fuzzy norm on vector space $E$ is an absolutely convex and absorbing fuzzy set $\rho$ with $ \inf_{t>0} (t\rho)(x)=0$, for $x\neq 0$.
\end{definition}
If $ \rho $ is a Katsaras norm on a vector space $ E, $ then the collection
$$ B_{\rho}= \lbrace \theta\wedge (t\rho)\: |\:\: t>0, 0<\theta \leq 1 \rbrace, $$ is a  base of neighborhoods of zero for a fuzzy linear topology on $E$. The fuzzy set $\mu$ is a fuzzy neighborhood of zero in this fuzzy topology if and only if
$$\exists \: 0<\theta\leq1,\:t>0\:;\: \theta\wedge (t\rho)\leq \mu.$$ Also $ \mu $ is a fuzzy neighborhood of $ x\in E $ if and only if $$ (x+\theta\wedge (t\rho))\leq \mu $$\\
i.e.
for $y\in E$, $$(x+\theta\wedge (t\rho))(y)=(\theta\wedge (t\rho))(y-x)\leq \mu(y). $$
In the following theorem, we prove that the topologies generated by Felbin-fuzzy norm and Katsaras norm are equivalent.
\begin{theorem}\label{th45}
If $ (E,\|\cdot\|) $ is a Felbin-fuzzy normed space, then there is a Katsaras norm $ \rho $ on $ E $ such that the fuzzy topologies generated by $ \|\cdot\| $ and $ \rho $ are equivalent.
\end{theorem}
\begin{proof}
Let $ \|\cdot\| $ be a Felbin-fuzzy norm on $E$. We define the fuzzy set $\rho$ on $E$ as follows:
$$\rho(x)=\begin{cases}1 &  \|x\|^{+}_{\alpha}< 1, \quad\\0 &else.\end{cases} $$
The fuzzy set $\rho$ is convex, since for $ t\in (0, 1]$,
\begin{align*}
(t\rho)(x)+((1-t)\rho)(x)&=\sup_{x=x_1+x_2}((t\rho)(x_1)\wedge((1-t)\rho)(x_2))\\
&=\sup_{x=x_1+x_2}\bigg(\begin{cases}1 &  \|x_1\|^{+}_{\alpha} < t \quad \\0 &else \end{cases}\wedge\begin{cases}1 &  \|x_2\|^{+}_{\alpha} <1-t \quad \\0 &else \end{cases}\bigg)\\
&\leq \begin{cases}1 &  \|x\|^{+}_{\alpha} < t+(1-t)=1 \quad \\0 &else \end{cases}\\
&=\rho(x).
\end{align*}

 Also $ \rho $ is balanced since for $t\neq 0$ with $\vert t \vert \leq 1$,
 \begin{align*}
(t\rho)(x)&=\rho(\dfrac{x}{t})\\
&=\begin{cases}1 &  \|x\|^{+}_{\alpha} < \vert t \vert \quad \\0 &else \end{cases}\\
&\leq \begin{cases}1 &  \|x\|^{+}_{\alpha} < 1 \quad \\0 &else \end{cases}\\
&=\rho(x).
\end{align*}

  Also we have $ \sup_{t>0}(t\rho)(x)=1 $ and $ \inf_{t>0}(t\rho)(x)=0. $ Then $ \rho $ is a Katsaras norm. Now, we prove that the fuzzy topologies generated by $ \|\cdot\| $ and $ \rho $ is equivalent. For $ 0< \alpha \leq 1 $ and $ \varepsilon >0 $ we have
  \begin{align*}
(\alpha\wedge (\varepsilon\rho))(y)&=\alpha\wedge \rho(\dfrac{y}{\varepsilon})\\
&=\alpha\wedge  \begin{cases}1 &  \|y\|^{+}_{\alpha} < \varepsilon \quad \\0 &else \end{cases}\\
&= \begin{cases}\alpha &  \|y\|^{+}_{\alpha} < \varepsilon \quad \\0 &else \end{cases}\\
&=\mu_{\alpha}(0,\varepsilon)(y).
\end{align*}

This shows that these two fuzzy topologies have a same base at zero, then these are equivalent.
\end{proof}
\begin{example}
Consider the Felbin-fuzzy normed space $ (\mathbb{R}^n,\|\cdot\|), $ where $ \|\cdot\|:\mathbb{R}^n \rightarrow F^{\ast}(\mathbb{R}) $ is defined as:
$$ \|(x_{1},x_{2},\cdots,x_{n})\|(t)= \begin{cases}1 &   t=\sqrt{x_{1}^{2}+x_{2}^{2}+\cdots +x_{n}^{2}}, \quad \\0 &else. \end{cases} $$
Then we have
 $$ \|(x_{1},x_{2},\cdots,x_{n})\|_{\alpha}^{+}=\sqrt{x_{1}^{2}+x_{2}^{2}+\cdots +x_{n}^{2}}. $$ Now, consider the Katsaras norm $$\rho(x)= \begin{cases}1 &   \Sigma^{n}_{i=1}x_{i}^{2} < 1, \quad \\0 &else. \end{cases} $$ Then by Theorem \ref{th45} the fuzzy topologies generated by $ \rho $ and $ \|\cdot\| $ are equivalent and for $ 0<\alpha \leq1 $ and $ \varepsilon >0$ we have $ \mu_{\alpha}(0,\varepsilon)(y)=(\alpha \wedge \varepsilon\rho)(y)$.
\end{example}

\section{Weak fuzzy topology}

\noindent

Let $X$ be a nonempty set, $\lbrace(Y_{\alpha}, \Gamma_{\alpha})\rbrace_{\alpha\in T}$
be a family of fuzzy topological spaces and let for each $\alpha\in T$, $f_{\alpha}: X \rightarrow Y_{\alpha}$ be a function. The weak fuzzy topology on $X$ generated by the family $ \mathcal{F}=\lbrace f_{\alpha} \rbrace_{\alpha\in T}$
is the coarsest or weakest fuzzy topology on $X$ that makes all the functions $f_{\alpha}$ fuzzy continuous. This fuzzy topology is generated by the fuzzy sets $\lbrace f^{-1}_{\alpha}(\mu): \:  \alpha\in T,\:  \mu\in \Gamma_{\alpha}\rbrace $, where the fuzzy sets $f^{-1}_{\alpha}(\mu)$ defined as follows:
$$  f^{-1}_{\alpha}(\mu)(x)= \mu(f_{\alpha}(x)) \quad \forall x\in X. $$
 Another subbase for the weak fuzzy topology consists of the sets
 $$ \lbrace f^{-1}_{\alpha}(\mu):\:  \alpha\in T,\:  \mu\in S_{\alpha}\rbrace ,$$ where $S_{\alpha}$ is a subbase for $\Gamma_{\alpha}$.
We denote the weak fuzzy topology on $X$ by $\sigma_{f}(X,\mathcal{F})$. A base for the weak fuzzy topology can be constructed as the collection of the sets of the form $\bigwedge^{n}_{k=1}f^{-1}_{\alpha_{k}}(\mu_{\alpha_{k}})$, where $\mu_{\alpha_{k}}\in \Gamma_{\alpha_{k}}$, $ \lbrace \alpha_{1},\alpha_{2},\cdots,\alpha_{n} \rbrace \subset T$.\\
\begin{example}
Suppose $ (X_{\alpha},\tau_{\alpha})_{\alpha \in T} $ is a family of fuzzy topological spaces, and $ X=\Pi_{\alpha \in T} X_\alpha. $ That is $ X=\lbrace f:T\rightarrow \cup X_{\alpha}:f(\alpha) \in X_{\alpha} \rbrace $. By the axiom of choice, we have $ X\neq \emptyset $. Let $ f \in X $, then $ f=(f_{\alpha})_{\alpha \in T} ,$ where $ f_\alpha=f(\alpha). $ Define $ P_{\alpha}:X\rightarrow X_{\alpha} $ by $ P_{\alpha}(f)=f_{\alpha}. $ Then the weak fuzzy topology generated by the family $ \mathcal{F}=\lbrace P_{\alpha}: \alpha \in T \rbrace $ has the subbasis $ \lbrace P^{-1}_{\alpha}(\mu_{\alpha}): \alpha \in T, \mu_{\alpha} \in \tau_{\alpha} \rbrace. $ But this is a subbasis for the product fuzzy topology on $ X $. Therefore $ \sigma_f(X,\mathcal{F}) $ coincides with the product topology on $X$.
\end{example}
 The following lemma is very important in the studing of weak fuzzy topologies.
\begin{lemma}\label{lkj}
A net satisfies $ x_{j}\xrightarrow{\sigma_{f}(X,\mathcal{F})} x $ for the weak fuzzy topology $ \sigma_{f}(X,\mathcal{F}) $ if and only if $ f_{\alpha}(x_{j})\longrightarrow f_{\alpha}(x) $ for each $ \alpha\in T $.
\end{lemma}
\begin{proof}
If we consider the topology $\sigma_{f}(X,\mathcal{F})$ on $X$, then the functions $f_{\alpha}$ are fuzzy continuous. This shows that if $ x_{j}\xrightarrow{\sigma_{f}(X,\mathcal{F})} x $, then $ f_{\alpha}(x_{j})\longrightarrow f_{\alpha}(x) $ for each $ \alpha\in T $. Conversely, let $ V= \bigwedge_{k=1}^n f^{-1}_{\alpha_{k}}(\mu_{\alpha_{k}}) $ be a weak fuzzy neighborhood of $x$, where $ \mu_{\alpha_{k}}\in \Gamma_{\alpha_{k}} $. For each $ k $, if $ f_{\alpha_{k}}(x_{j})\rightarrow f_{\alpha_{k}}(x) $, then for each $0<r_k<\mu_{\alpha_k}(f_{\alpha_k}(x))$ there is $j_k$ such that for each $ j>j_k $ we have $ \mu_{\alpha_{k}}(f_{\alpha_{k}}(x_{j}))>r_k $. Therefor for $j> max\{j_1,j_2,\cdots,j_n\}=j'$, we have $$f^{-1}_{\alpha_{k}}(\mu_{i_{k}})(x_{j})>max\{r_1,r_2,\cdots,r_n\}=r.$$
 Then $V(x_j)>r$ for each $j>j'$. This proves our claim.
\end{proof}
\begin{remark}
On $\mathbb{K}$, if we consider the fuzzy topology is obtained from the usual topology of $ \mathbb{K} $ that is the family of fuzzy sets $ \mu:\mathbb{K} \rightarrow I, $ which are lower semi-continuous, then endowed with this topology $ \mathbb{K} $ is a fuzzy topological vector space. This topology is compatible with the fuzzy norm
$$\rho(x)=\begin{cases}1 &|x|<1,\\0 & else.\end{cases}$$
For this, we prove that the collection $B_\rho=\{\theta\wedge (t\rho) \:| \:\: t>0,0<\theta\leq 1\}$ is a base of fuzzy neighborhoods for zero. Let $\mu$ be a neighborhood of zero in $\mathbb{K}$. Then $\mu$ is lower semi-continuous and $\mu(0)>0$. Let $0<r<\mu(0)$. Then $\mu^{-1}(r,1]$ is open in $\mathbb{K}$ and $0\in\mu^{-1}(r,1]$. Then there is $m>0$ such that $$\{x\in\mathbb{K}:|x|<m\}\subseteq\mu^{-1}(r,1].$$
Now, we claim that $r\wedge(m\rho)\leq \mu$. Indeed, for $x\in\mathbb{K}$, we have
$$r\wedge(m\rho)(x)=r\wedge\rho(\frac{x}{m})=\rho(x)=\begin{cases}r &|x|<m,\\0 & else.\end{cases}$$ Then, we have $r\wedge(m\rho)(x)\leq\mu(x)$.
Also this topology is compatible with the Flebin-fuzzy norm $\Vert .\Vert:\mathbb{K}\rightarrow F^{\ast}(\mathbb{R})$, defined for $r\in\mathbb{K}$ as follows:
$$\Vert x\Vert (t)=\begin{cases}1 &t=|x|,\\0 & else,\end{cases}$$
with $L=\min$ and $R=\max$. But this topology is not the only fuzzy Hausdorff linear topology on $\mathbb{K}$. For example it was showen in \cite{g} that the fuzzy norm $$\Vert x \Vert^{\ast}(t)=\begin{cases}1-\frac{t}{|x|} &0 \leq t \leq |x|, x\neq0, \\0 & else,\end{cases}$$ is not equivalent with the fuzzy norm
$$\Vert x\Vert (t)=\begin{cases}1 &t=|x|,\\0 & else.\end{cases}$$ Then, we can consider more than one topology on $\mathbb{K}$.
\end{remark}
\begin{definition}
 Let $(E,\tau)$ be a fuzzy topological vector space and $ \Gamma $ be locally convex Hausdorff linear topology on $ \mathbb{K} $. We denote the family of all fuzzy continuous linear functionals $f: (E,\tau)\rightarrow (\mathbb{K},\Gamma) $ by $ E^{\prime}_{\Gamma} $ and call it the $ \Gamma- $ fuzzy dual of $E$. Also, we denote the algebraic dual of $ E $ by $ E^{\ast} $.
\end{definition}
In \cite{ro}, the concept of dual pair defined as follows:\\
We call $(E,E')$ a dual pair, whenever $E$ and $E'$ are two vector spaces over the same $ \mathbb{K} $ scaler field and  $ \langle x, x' \rangle $ is a bilinear form on $E$ and $ E' $ satisfying the following conditions:
\begin{enumerate}
\item[$(D)$] For each $ x\neq0 $ in $ E $, there is $ x^{\prime}\in E^{\prime} $ such that $ \langle x, x^{\prime}\rangle\neq0 $.\\
\item[$(D^{\prime})$] For each $ x^{\prime}\neq0 $ in $ E^{\prime} $ there is $ x\in E $ such that $ \langle x, x^{\prime}\rangle\neq0 $.\\
\end{enumerate}
\begin{definition}
Let $ (E,E') $ be a dual pair. Then each $x'\in E'$ is a linear functional on $ E $ by taking:
$$ x'(x)=\langle x, x'\rangle \:\:\:\:\:\: x\in E. $$
Consider on $ \mathbb{K} $ a locally convex Hausdorff linear topology $ \Gamma .$ We denote the weak fuzzy topology on $ E $ generated by mapping $ x': E\rightarrow \mathbb{K} $ by $ \sigma^{\Gamma}_{f}(E,E'). $
\end{definition}
We note that if $ \Gamma_1 $ and $ \Gamma_2 $ are two topologies on $ \mathbb{K} $ and $ \Gamma_1 \subset \Gamma_2 $, then $ \sigma^{\Gamma_1}_{f}(E,E') \subset \sigma^{\Gamma_2}_{f}(E,E'). $
\begin{theorem}
Let $(E,E')$ be a dual pair. Then $E$ endowed with $\sigma^{\Gamma}_f(E,E')$ is a locally convex fuzzy topological vector space.
\end{theorem}
\begin{proof}
We prove that the addition and scaler multiplication are continuous. For $x,y,z\in E$, let $z=x+y$. Also suppose $V=\bigwedge_{i=1}^n f_i^{-1}(\mu_i)$ is a weak fuzzy neighborhood of $z$, where $f_i\in E'$ and $\mu_i$ is a fuzzy neighborhood in $ \mathbb{K} $. Then $V(z)>0$. This shows that $\mu_i(f_i(z))=f_i^{-1}(\mu_i)(z)>0$ for each $i=1,2,\cdots,n$. Then for each $i=1,2,\cdots,n$, $\mu_i$ is a fuzzy neighborhood of $f_i(z)$ in $\mathbb{K}$. Now since
we have $f_i(x)+f_i(y)=f_i(z)$ and $\mathbb{K}$ is a fuzzy topological vector space, for each $i=1,2,\cdots,n$, there is fuzzy neighborhoods $\omega_i$ and $\psi_i$ of $f_i(x)$ and $f_i(y)$ (respectively) such that if $\omega_i(a)>0$ and $\psi_i(b)>0$, then $\mu_i(a+b)>0$ and $\omega_i+\psi_i\leq\mu_i$. Now we set
\begin{center}
$U=\bigwedge_{i=1}^n f_i^{-1}(\omega_i)$ and $O=\bigwedge_{i=1}^n f_i^{-1}(\psi_i)$.
\end{center}
Then, we have
$$U(x)=\bigwedge_{i=1}^n f_i^{-1}(\omega_i)(x)=\bigwedge_{i=1}^n \omega_i(f_i(x))>0,$$
 and
  $$O(y)=\bigwedge_{i=1}^n f_i^{-1}(\psi_i)(y)=\bigwedge_{i=1}^n \psi_i(f_i(y))>0.$$
  Therefore $U$ and $V$ are weak fuzzy neighborhoods of $x$ and $y$ respectively. Now, let for $m,n\in E$, $U(m)>0$ and $O(n)>0$. Then for each $i=1,2,\cdots,n$, we have $\omega_i(f_i(m))>0$ and $\psi_i(f_i(n))>0$ and therefore, $\mu_i(f_i(m+n))>0$. Thus
  \begin{align}
  V(m+n)&=\bigwedge_{i=1}^n f_i^{-1}(\mu_i)(m+n)\nonumber\\&=\bigwedge_{i=1}^n \mu_i(f_i(m+n))>0\nonumber.
  \end{align}
   Also, since $\omega_i+\psi_i\leq\mu_i$ for all $i=1,2,\cdots,n$, we have
  \begin{align}
  (U+O)(x)&=\sup_{x=x_1+x_2}\bigwedge_{i=1}^{n}(\omega_i(f_i(x_1))\wedge\psi_i(f_i(x_2)))
  \nonumber\\ &=\bigwedge_{i=1}^{n}\sup_{x=x_1+x_2}(\omega_i(f_i(x_1))\wedge\psi_i(f_i(x_2)))\nonumber \\&= \bigwedge_{i=1}^{n}(\omega_i+\psi_i)(f_i(x))\nonumber \\&\leq\bigwedge_{i=1}^{n}\mu_i(f_i(x))\nonumber\\&=
  V(x).
  \end{align}
   Therefore $U+O\leq V$. Then the addition is fuzzy continuous with respect to $\sigma^{\Gamma}_f(E,E')$.

  Now, we show that the scaler multiplication is continuous. For $t\in \mathbb{K} $ and $x\in E$, let $M= \bigwedge_{i=1}^n f_i^{-1}(\gamma_i)$ is a weak fuzzy neighborhood of $tx$. Then for each $i=1,2,\cdots,n$, we have $$\gamma_i(tf_i(x))=\gamma_i(f_i(tx))=f_i^{-1}(\mu_i)(tx)>0.$$
   Then $\mu_i$ is fuzzy neighborhood of $tf_i(x)$ in $ \mathbb{K} $. Since $ \mathbb{K} $ is a fuzzy topological vector space, there is fuzzy neighborhood $\eta_i$ and $\vartheta_i$ of $t$ and $f_i(x)$ in $ \mathbb{K} $ such that if for $a,b\in \mathbb{K} $, $\eta_i(a)>0$ and $\vartheta_i(b)>0$ then $\mu_i(ab)>0$ and $a\vartheta_i\leq \mu_i$. Now, we set
  $$S=\bigwedge_{i=1}^n f_i^{-1}(\vartheta_i).$$
  We have
   $$S(x)=\bigwedge_{i=1}^n f_i^{-1}(\vartheta_i)(x)=\bigwedge_{i=1}^n\vartheta_i(f_i(x))>0.$$
  Then $S$ is a weak fuzzy neighborhood of $x$. Now, let for $a\in \mathbb{K}$, $\eta_i(a)>0$ and for $y\in S$, $S(y)>0$, then $\vartheta_i(f_i(y))>0$ for each $i=1,2,\cdots,n$. The above arguments show that $f_i^{-1}(\mu_i)(ay)=\mu_i(af_i(y))>0$ for each $i$. Then $V(ay)=\bigwedge_{i=1}^n f_i^{-1}(\mu_i)(ay)>0$. Also for $a\neq 0$ with $\eta_i(a)>0$, we have
  \begin{align}
   (aS)(y)&=\bigwedge_{i=1}^n f_i^{-1}(\vartheta_i)(\frac{y}{a})\nonumber \\ &=\bigwedge_{i=1}^n (\vartheta_i(f_i(\frac{y}{a}))\nonumber \\ &=\bigwedge_{i=1}^n (a\vartheta_i(f_i(y))\leq \mu_i(y).\nonumber
  \end{align}
   Then the scaler multiplication is fuzzy continuous. Now, we prove that $ \sigma^{\Gamma}_f(E,E') $ is locally convex. Since $ (\mathbb{K},\Gamma) $ is locally convex, $ \Gamma $ has a base of convex neighborhoods. Let $ \mu \in \Gamma $ be convex. It is enough to show that $ x'^{-1}(\mu) $ is convex for $ x' \in E'. $ For $ t \in I $ and $ x,y \in E $ we have
   \begin{align}
   x'^{-1}(\mu)(tx+(1-t)y)&=\mu(x'(tx+(1-t)y)) \nonumber \\ &=\mu(tx'(x)+(1-t)x'(y)) \nonumber \\ &\leq t\mu(x'(x))+(1-t)\mu(x'(y)) \nonumber \\ &= tx'^{-1}(\mu)(x)+(1-t)x'^{-1}(\mu)(y). \nonumber
  \end{align}
   Therefore $ x'^{-1}(\mu) $ is convex.
  \end{proof}
  \begin{remark}
 We note that since in a fuzzy topological vector space the addition and scaler multiplication are continuous, one can use local basis instead of basis for a fuzzy linear topology.
\end{remark}
If $(E,E^{\prime})$ is a dual pair and $S$ is a local base in $ \mathbb{K} $, the sets $\bigwedge^{n}_{i=1}x^{-1}_{i}(\mu_{i})$ where $ x_{1},x_{2},\cdots,x_{n}\in E^{\prime} $ and $ \mu_{1},\mu_{2},\cdots,\mu_{n}\in S $ consist a base for $\sigma^{\Gamma}_{f}(E,E^{\prime})$. If we consider the usual topology on $\mathbb{K}$ which is compatible with the fuzzy norm
$$\rho(x)=\begin{cases}1 &|x|<1,\\0 & else,\end{cases}$$ then, the collection $$B_\rho=\{\theta\wedge (t\rho)\: |\:\: t>0,0<\theta\leq 1\},$$ is a base of fuzzy neighborhood of zero in $\mathbb{K}$. This shows that the collection
$$B_\rho=\{\bigwedge_{i=1}^n f_i^{-1}(\theta_i\wedge (t_i\rho))\: |\:\:t_i>0,0<\theta_i\leq 1, f_i\in E'\},$$ is a base of fuzzy neighborhood of zero for $\sigma^{\Gamma}_f(E,E')$.
\begin{lemma}\cite{ro} \label{l478}
If $f_0,f_1,\cdots,f_n$ are linear functionals on a vector space $E$, then either $f_0$ is a linear combinations of $f_1,f_2,\cdots,f_n$, or there is $a\in E$ such that $f_0(a)=1$ and $$f_1(a)=f_2(a)=\cdots=f_n(a)=0.$$
\end{lemma}
\begin{proof}
See \cite{ro}, Chapter II, Section 3.
\end{proof}
\begin{proposition}
If $(E,E') $ is a dual pair, then the fuzzy dual of $ E $ under $ \sigma^{\Gamma}_{f}(E,E') $ is $ E' $.
\end{proposition}
\begin{proof}
Let $ f $ be a linear form on $ E $ continuous under $ \sigma^{\Gamma}_{f}(E,E') $ and let $ S $ be a base in $ \mathbb{K} $. For each $ \mu\in S $ with $ \mu(0)>\mu(1) $, there is $ x^{\prime}_{1},x'_{2},\cdots,x'_{n}\in E^{\prime} $ such that $$ \bigwedge^{n}_{i=1}x^{-1}_{i}(\mu)\leq f^{-1}(\mu) $$ and $$ \bigwedge^{n}_{i=1}x^{-1}_{i}(\mu)(0)=f^{-1}(\mu)(0)>0 $$ therefor for $ x\in E,  \mu(x_{i}(x))\leq\mu(f(x)) $. Now by Lemma \ref{l478}, either $ f $ is a linear combination of $ x^{\prime}_{1},x^{\prime}_{2},\cdots,x^{\prime}_{n} $ or there is some $ a\in E $ such that $ f(a)=1 $ but $ x'_i(a)=0 $ for all $i=1,2,\cdots,n$. This shows that $ \mu(0)\leq\mu(1) $ and this contradiction. Then
$$ f=\sum\lambda_{i}x^{\prime}_{i}\in E^{\prime}$$ on the other hand if $ x^{\prime}\in E^{\prime} $ then it is continuous under $ \sigma^{\Gamma}_{f}(E,E^{\prime}) $ by the definition of weak fuzzy topology.
\end{proof}
\begin{theorem}
Let $E$ be a vector space and $E^*$ be its algebraic dual. Then $E^*$ is complete under the fuzzy topology $\sigma^{\Gamma}_f(E^*,E)$.
\end{theorem}
\begin{proof}
We consider $E^*$ endowed with the weak star fuzzy topology $\sigma^{\Gamma}_f(E^*,E)$. Suppose $(f_j)_{j\in T}$ is a fuzzy Cauchy net in $E^*$. Then for each $a\in E$ the net $(f_j(a))_{j\in T}$ is a fuzzy Cauchy net in $\mathbb{K}$. Since $ \mathbb{K} $ is complete under its usual fuzzy topology, for each $a\in E$ there is $f(a)\in\mathbb{K}$ such that $f_j(a)\rightarrow f(a)$. Since each $f_j$ is linear, then $f$ is a linear functional on $E$. Now Lemma \ref{lkj} shows that $ f_{j}\xrightarrow{\sigma^{\Gamma}_f(E^*,E)} f$. Then $E^*$ is complete under the fuzzy topology $\sigma^{\Gamma}_f(E^*,E)$.
\end{proof}
\begin{definition}
If $ (E,E')$ is a dual pair, any fuzzy topology on $E$ under which the dual of $ E $ is $ E' $ is called the {\it fuzzy topology of dual pair}.
\end{definition}
In a special case if $ (E,\tau) $ is a Hausdorff fuzzy topological vector space, then $ (E,E'_{f}) $ is a dual pair and $ \tau $ is a topology of dual pair.
Also,
 $\sigma^{\Gamma}_{f}(E,E')$ is the coarsest fuzzy topology of dual pair $(E,E')$.
\begin{theorem}\label{t159}
If $ (E,E') $ is a dual pair, then $ (E,\sigma^{\Gamma}_{f}(E,E')) $ is fuzzy Hausdorff.
\end{theorem}
\begin{proof}
Let $ x,y\in E $ and $ x\neq y $. Then by condition $ (D) $ there is $ x'\in E' $ such that $ x'(x-y)\neq0 $ i.e. $ x^{\prime}(x)\neq x^{\prime}(y) $. Since $ \mathbb{R} $ is Hausdorff there is fuzzy open sets $ \beta, \eta $ in $ \mathbb{R} $ such that $ \beta\wedge\eta=0 $ and $ \beta(x^{\prime}(x))=\eta(x^{\prime}(y))=1 $. Now we consider the fuzzy open sets $ x^{\prime^{-1}}(\eta), x^{\prime^{-1}}(\beta) $ in $ E $. These sets are the neighborhoods of $ x $ and $ y $ and we have
$$(x^{\prime^{-1}}(\eta)\wedge x^{\prime^{-1}}(\beta))(a)= min \lbrace \eta(x^{\prime}(a)), \beta(x^{\prime}(a))\rbrace=0$$
 and
$$ x^{\prime^{-1}}(\eta)(x)= \eta(x^{\prime}(x))=1 $$
similarly,
$$ x^{\prime^{-1}}(\beta)(y)= \beta(x^{\prime}(y))=1 .$$
Now the proof is complete.
\end{proof}
\begin{corollary}
Let $(E,E')$ be a dual pair and $\tau$ be a fuzzy dual pair topology on $E$. Then $E$ is Hausdorff endowed with $\tau$.
\end{corollary}
\begin{proof}
Clearly, $\tau$ is finer than $\sigma^{\Gamma}_f(E,E')$. Now Theorem \ref{t159} shows that $E$ is Hausdorff endowed with $\tau$.
\end{proof}
\begin{proposition}\label{145}
If $ (E,\tau) $ is finite dimensional Hausdorff fuzzy topological vector space, then $ E^{\prime}_{\Gamma}=E^{\ast}$.
\end{proposition}
\begin{proof}
Let $B=\{e_1,e_2,\cdots,e_n\}$ be a base for $E$. Then there is a base $B'=\{e_1^*,e_2^*,\cdots,e_n^*\}$ for $E^*$ such that $e_i^*(e_i)=1$ for each $i=1,2,\cdots,n$ and $e_i^*(e_j)=0$ for $i\neq j$. Then $E$ and $E^*$ are algebraical isomorphic. Since $(E, E'_\Gamma)$ is a dual pair, we have $E\subseteq (E'_\Gamma)^*\simeq E'_\Gamma$ and $E'_\Gamma \subseteq E^*\simeq E$. Then $E'_\Gamma=E^*$.
\end{proof}
\begin{remark}
Let $E$ be finite dimensional vector space. We know that there is only a unique Hausdorff linear topology on $E$. But this is not true in the fuzzy case i.e. the Hausdorff fuzzy linear topology on $E$ is not unique. For example consider the vector space $\mathbb{R}$. The collection $\tau_1$ of all lower semi-continuous fuzzy sets is a Hausdorff fuzzy linear topology. Also, the collection $\tau_2$ of a all characteristic functions of open subsets of $ \mathbb{R} $ is a  Hausdorff fuzzy linear topology. But these two topologies are not identical and we have $\tau_2\subset\tau_1$.
\end{remark}
\begin{corollary}
Let $(E,\tau)$ be a finite dimensional Hausdorff fuzzy topological vector space. Then we have $\sigma^{\Gamma}_f(E,E'_\Gamma)=\sigma^{\Gamma}_f(E,E^*)$. In fact this is concluded form Proposition \ref{145}.
\end{corollary}
It is well known that between all the linear topologies on a finite dimensional
vector space $E$ there is a strongest one, namely the unique Hausdorff linear
topology on $E$. We denote this topology by $\tau$. Let $\omega(\tau)$ be the collection of all lower semi-continuous from $E$ into $I$. Then $\omega(\tau)$ is a fuzzy linear topology on $E$. We denote the usual weak topology on $E$ by $\sigma(E,E^*)$.
\begin{theorem}\label{toN}
Let $(E,E')$ be a dual pair. If we consider the usual topology $ \omega(\tau) $ on $\mathbb{R}$, then
 $$\sigma^{\omega(\tau)}_f(E,E')=\omega(\sigma(E,E')).$$
\end{theorem}
\begin{proof}
It is enough to show that for $ f_{1}, f_{2},\cdots, f_{n}\in E' $ and $ \mu_{1}, \mu_{2},\cdots , \mu_{n}\in \omega(\tau) $ the fuzzy set $ \bigwedge^{n}_{i=1}f^{-1}_{i}(\mu_{i}) $ is lower semi-continuous on $ E $ with respect to the topology $ \sigma(E,E'). $ Since $ \mu_{i}\in \omega(\tau) $ for $ i=1,2,\cdots, n $, then $ \mu_{i}: \mathbb{K}\rightarrow I $ are lower semi-continuous. Also $$ f_{i}: (E,\sigma(E,E'))\rightarrow \mathbb{K} $$ is continuous and then is lower semi-continuous. Then $\mu_{i} \circ f_{i} $ is lower semi-continuous for each $ i=1,2,\cdots,n. $ Then $$ \bigwedge^{n}_{i=1}f^{-1}_{i}(\mu_{i})= \bigwedge^{n}_{i=1}\mu_{i} \circ f_{i} $$ is lower semi-continuous on $ (E,\sigma(E,E')) $ for each $ i=1,2,\cdots,n. $
\end{proof}
If we consider on $\mathbb{K}$ any fuzzy Hausdorff linear topology different from the topology $ \omega(\tau), $ then the topologies $ \sigma^{\omega(\tau)}_f(E,E') $ and $ \omega(\sigma(E,E')) $ are not equivalent. We illustrate this fact by an example.
\begin{example}\label{tor}
Consider the vector space $ \mathbf{C}(\Omega) $, the collection of all real-valued continuous functions on the open set $ \Omega \subseteq \mathbb{R}^{n}.$ For each $ x\in \Omega $ we set $ \delta_{x}:\mathbf{C}(\Omega)\rightarrow \mathbb{R}: \delta_{x}(f)=f(x). $ Then $ \delta_{x} $ is a linear functional on $ \mathbf{C}(\Omega) $. Let $ V=span \lbrace \delta_{x}: x\in \Omega \rbrace $. Then $ V $ is a real vector space and $ (\mathbf{C}(\Omega),V) $ is a dual pair. On $ \mathbb{R} $ the fuzzy norms
$$\Vert x \Vert^{\ast}(t)=\begin{cases}1-\frac{t}{|x|} &0 \leq t \leq |x|, x\neq0, \\0 & else,\end{cases}$$ and $$\Vert x\Vert (t)=\begin{cases}1 &t=|x|,\\0 & else,\end{cases}$$ are not equivalent (See \cite{g}, Remark 3.1). The fuzzy norm $ \Vert.\Vert $ induces the fuzzy topology $ \omega(\tau) $ which is the finest fuzzy topology on $ \mathbb{R} $ (See \cite{k7}, Theorem 3.18). The fuzzy norm $ \Vert.\Vert^{\ast} $ induces the fuzzy topology $ \mathcal{T}^{\ast}_{\Vert.\Vert} $ on $ \mathbb{R} $  which is weaker than $ \omega(\tau). $ Then there is $ \mu\in \omega(\tau) $ such that $ \mu\notin \mathcal{T}^{\ast}_{\Vert.\Vert}. $ Now by Theorem \ref{toN}, we have $$ \bigwedge^{n}_{i=1}\delta^{-1}_{x_{i}}(\mu)\in \omega(\sigma(\mathbf{C}(\Omega),V)) $$ for $ x_{1}, x_{2},\cdots, x_{n}\in \Omega $. But if we consider $ \mathcal{T}^*_{\Vert.\Vert} $ on $ \mathbb{R} $, then we have $$ \bigwedge^{n}_{i=1}\delta^{-1}_{x_{i}}(\mu)\notin \sigma^{\mathcal{T}^{\ast}_{\Vert.\Vert}}_{f}(\mathbf{C}(\Omega),V) $$  for $ x_{1}, x_{2},\cdots, x_{n}\in \Omega $, since  $ \mu\notin \mathcal{T}^{\ast}_{\Vert.\Vert}. $ This show that $ \sigma^{\mathcal{T}^{\ast}_{\Vert.\Vert}}_{f}(\mathbf{C}(\Omega),V) $ is strictly weaker than $ \omega(\sigma(\mathbf{C}(\Omega),V)). $
\end{example}
Theorem \ref{toN} and Example \ref{tor} show that the weak fuzzy topology is an extension of weak topology.\\
Let $T : E \rightarrow F$ be a linear operator between two vector spaces.
 Every $y^*\in F^*$ gives rise to
a real function $T^*y^*$ on $E$ defined pointwise via the formula
$$T^*y^*(x)= y^*\circ T(x)=y^*(T(x)),$$ for $x\in E$.
Clearly $T^*y^*$ is linear and so belongs to $X^*$
The operator $T^*$ is called the {\it algebraic adjoint} of $T$.
\begin{definition}
Let $(E,\tau)$ and $(F,\xi)$ be fuzzy topological vector spaces. The linear operator is called weakly fuzzy continuous whenever  $$T:(E,\sigma^{\Gamma}_f(E,E'_\Gamma)\rightarrow (F,\sigma^{\Gamma}_f(F,F'_\Gamma)),$$ is fuzzy continuous.
\end{definition}
\begin{lemma}
Let $(E,\tau)$ be a fuzzy topological vector space. Then for a net $(x_i)_{i\in S}$ we have $x_i\xrightarrow{\sigma^{\Gamma}_f(E,E')}x$ if and only if $x_i-x\xrightarrow{\sigma^{\Gamma}_f(E,E')}0$.
\end{lemma}
\begin{proof}
Let $x_i\xrightarrow{\sigma^{\Gamma}_f(E,E')}x$. Then for each $T\in E'$, we have $$T(x_i)\xrightarrow{\sigma^{\Gamma}_f(E,E')}T(x).$$ Then for each neighborhood $\mu$ of zero and in $\mathbb{K}$ each $0<r<\mu(0)$ there is $i_0\in S$ such that for every $i>i_0$, $(T(x)+\mu)(T(x_i))>r$. This shows that $\mu( T(x_i)-T(x))>r$. Since $T$ is linear, we have $\mu( T(x_i-x))>r$. It follows that $x_i-x\xrightarrow{\sigma^{\Gamma}_f(E,E')}0$. The converse is similar.
\end{proof}
The next result offers a very simple criterion for deciding whether a linear
operator is weakly continuous. You only have to check that its adjoint carries
continuous functionals into continuous functionals.
\begin{theorem}\label{thgm1}
Let $(E,E')$ and $(F,F')$ be
dual pairs and let $T : E \rightarrow F$ be a linear
operator, where $E$ and $F$ are endowed with their weak fuzzy topologies. Then $T$ is
weakly fuzzy continuous if and only if the algebraic adjoint $T^*$ satisfies $T^*(F_\Gamma') \subseteq E_\Gamma'$.
\end{theorem}
\begin{proof}
Firstly, suppose that $T$ is weakly fuzzy continuous. Since each $y'\in F'_\Gamma$ is weakly fuzzy continuous, then $y'\circ T$ is weakly fuzzy continuous linear operator on $E$. This shows that $y'\circ T\in E'_\Gamma$. Since $T^*(y')=y'\circ T$ , then $T^*(y')\in E'_\Gamma$.

Conversely, let  $T^*(F_\Gamma') \subseteq E_\Gamma'$ and $(x_i)_{i\in S}$ be a net in $E$ such that $x_i\xrightarrow{\sigma_f(E,E'_\Gamma)} x$. This shows that for each $f\in E'_\Gamma$, $f(x_i)\rightarrow f(x)$. Now since $T^*(F_\Gamma') \subseteq E_\Gamma'$, then for each $y'\in F'_\Gamma$, we have $T^*(y')(x_i)\rightarrow T^*(y')(x_i)$. Therefore  $y'(T(x_i))\rightarrow y'(T(x))$ for each $y'\in F'_\Gamma$. This shows that $T(x_i)\xrightarrow{\sigma_f(F,F'_\Gamma)} T(x)$. Therefore $T$ is weakly fuzzy continuous.
\end{proof}
\begin{theorem}\label{km123}
Let $E$ and $F$ be Hausdorff locally convex fuzzy topological vector spaces and $T:E\rightarrow F$ be a fuzzy continuous linear operator. Then $T$ is weakly continuous.
\end{theorem}
\begin{proof}
Suppose $T:E\rightarrow F$ is fuzzy continuous. Then for each $y'\in F'_\Gamma$, $T^*(y')=y'of$ is a fuzzy continuous linear operator on $E$ and then $T^*(y')\in E'_\Gamma$. Therefore  $T^*(F'_\Gamma)\subseteq E'_\Gamma$. Now, Theorem \ref{thgm1} shows that $T$ is weakly continuous.
\end{proof}
Theorem \ref{km123} shows that every fuzzy continuous linear operator is weakly continuous. In the following theorem, we consider some conditions under which every weakly fuzzy continuous linear operator in continuous.
\begin{theorem}
Let $E$ be a bornological fuzzy topological vector space and $F$ be a fuzzy topological vector space such that every weakly bounded fuzzy sets in $F$ is bounded. The every weakly continuous linear operator $T:E\rightarrow F$ is continuous.
\end{theorem}
\begin{proof}
Since $E$ is bornological, it is enough show that $T$ is bounded. Let $\psi$ be a bounded fuzzy set in $E$. Then $\psi$ is weakly bounded. Since $T$ is weakly fuzzy continuous, $T(\psi)$ is a weakly fuzzy bounded set in $F$. Now that assumption of theorem shows that $T(\psi)$ is a bounded fuzzy set in $F$. Then $T$ is bounded.
\end{proof}
\begin{corollary}
Every linear functional on a bornological fuzzy linear space is continuous if and only if it is fuzzy continuous, Since the weak fuzzy topology on $\mathbb{K}$ is identical with the original topology.
\end{corollary}
\begin{corollary}
Every linear functional on a seminormed fuzzy linear space is continuous if and only if it is fuzzy continuous, Since every seminormed fuzzy linear space is bornological.
\end{corollary}
\begin{example}
Let $E$ be vector space and $\mathcal{F}_{0}$ be the collection of all absolutely convex and absorbent fuzzy sets in $E$. Then $\mathcal{F}_{0}$ satisfies the conditions of Theorem 4.2 from \cite{k7}. Then by Theorem 4.2 from \cite{k7}, there exists a fuzzy linear topology $\tau_\beta$ on E such that $\mathcal{F}_{0}$
coincides with the family of all fuzzy neighborhoods of zero. The fuzzy topological space $ (E,\tau_\beta) $ is fuzzy Hausdorff. Indeed, if $ a,b \in E $ and $ a\neq b, $ then the fuzzy points $ a_1 $ and $ b_1 $ are fuzzy neighborhoods of $ a $ and $ b $ respectively, and we have $ a_1 \wedge b_1=0. $ Then $ (E,\tau_\beta) $ is fuzzy Hausdorff. Also for $ (E,\tau_\beta) $ we have $E'_\Gamma=E^*$. In fact, for $f\in E^*$ and each $ \mu \in \Gamma $ the fuzzy set $ f^{-1}(\mu) $ is absolutely convex and absorbing and then $ f^{-1}(\mu) \in \mathcal{F}. $ Therefore $ f \in E'_{\Gamma}. $ In this example, we have $\sigma^{\Gamma}_f(E,E'_\Gamma)=\sigma^{\Gamma}_f(E,E^*)$.
\end{example}
The weak fuzzy bounded subsets of a fuzzy topological vector space have an important role in the constructing of fuzzy topologies on the dual space.
\begin{theorem}\label{thm48}
Let $(E,E')$ be a dual pair. Then the fuzzy set $\mu$ in $E$ is weakly bounded if and only if $$\varphi_\mu(x')=\sup_{y\in \mathbb{K}}x'(\mu)(y),$$ is a fuzzy seminorm on $E'$.
\end{theorem}
\begin{proof}
Let $\mu$ be weakly  fuzzy bounded. Then for each $x'\in E'$, $x'(\mu)$ is fuzzy bounded set in $\mathbb{K}$. This shows that $\varphi_\mu$ is well defined. Since $\mu$ is convex and $x'$ is linear then $\varphi_\mu$ is convex.
We have $$\varphi_\mu(0)=\sup_{y\in \mathbb{K}}0(\mu)(y)=\sup_{x\in E}\mu(x)=1.$$
 This shows that $\varphi_\mu$ is absorbing. For $t\in\mathbb{K}$ with $|t|\leq 1$, we have
\begin{align}
\varphi_\mu(tx')&=
\sup_{y\in \mathbb{K}}(tx')(\mu)(y)\nonumber\\&=\sup_{y\in \mathbb{K}}\sup_{x\in (tx')^{-1}(y)}\mu(x)\nonumber\\&=\sup_{y\in \mathbb{K}}
\sup_{x\in x'^{-1}(\frac{y}{t})}\mu(x)\nonumber\\&=\sup_{y\in \mathbb{K}}\sup_{x\in x'^{-1}(y)}(\mu)(x)
\nonumber\\&=\varphi_\mu(x').\nonumber
\end{align}
This shows that $\varphi_\mu$ is balanced. The converse is clear.
\end{proof}


\section{Conclusion}

It is well known that there is only one linear Hausdorff topology on $ \mathbb{K} $ (generally on finite dimensional spaces) but this is not true in fuzzy structure. This leads to some difference between weak topology and weak fuzzy topology. But, we proved that when $ \Gamma =\omega(\tau), $ we have $ \sigma^{\Gamma}_{f}(E,E')=\omega (\sigma(E,E')). $ Then, we can consider $ \sigma(E,E') $ as an special case of $ \sigma^{\Gamma}_{f}(E,E'). $ Also, the difference between weak topology and weak fuzzy topology leads to new results on the theory of topological vector spaces. The weak fuzzy topology is very useful in constructing topologies on dual spaces. In the future, we concentrate for proving the Mackey-Arens and Banach-Alaoglu theorems in the fuzzy topological vector spaces using from weak fuzzy topology.

%

\date{\scriptsize $^{a}$
E-mail:bdaraby@maragheh.ac.ir,}\\
\date{\scriptsize $^{b}$
E-mail:nasibeh$_{-}$khosravi@yahoo.com,}\\
\date{\scriptsize $^{c}$
E-mail:rahimi@maragheh.ac.ir\\

}

\end{document}